\documentclass[12pt,reqno]{amsart}
\usepackage{amssymb,comment}
\usepackage{hyperref}
\usepackage{stmaryrd}
\usepackage{slashed}
\usepackage{url}

\newtheorem{theorem}{Theorem}
\newtheorem{sat}[theorem]{Proposition}
\newtheorem{lem}[theorem]{Lemma}

\theoremstyle{remark}
\newtheorem*{bem}{Remark}

\numberwithin{theorem}{section}
\numberwithin{equation}{section}
   
\newcommand{\R}{\mathbb{R}}

\newcommand{\C}{\mathbb{C}}
\newcommand{\N}{\mathbb{N}}
\newcommand{\Z}{\mathbb{Z}}

\newcommand{\id}{Id}

\def\dim{\operatorname{dim}}

\newcommand{\tr}{tr}

\renewcommand{\part}[2]{\frac{\partial #1}{\partial #2}}

\begin{document}

\title{On Conformal Powers of the Dirac Operator on Einstein Manifolds}

\author[Matthias Fischmann, Christian Krattenthaler and Petr Somberg]
{Matthias Fischmann$^1$, Christian Krattenthaler$^2$ and Petr Somberg$^1$}

\address{
$^1$ E. \v{C}ech Institute, Mathematical Institute of Charles University,
Sokolovsk\'a 83, Praha 8 - Karl\'{\i}n, Czech Republic.
\newline\indent
$^2$ Fakult\"at f\"ur Mathematik, Universit\"at Wien, 
Oskar-Morgenstern-Platz 1, A-1090 Vienna, Austria.
}


\thanks{$^1$
Research partially supported by
grant GA~CR~P201/12/G028. \newline\indent
$^2$
Research partially supported by the Austrian
Science Foundation FWF, grants Z130-N13 and S50-N15,
the latter in the framework of the Special Research Program
``Algorithmic and Enumerative Combinatorics."
}

\date{}

\begin{abstract}
We determine the structure of conformal powers of the Dirac operator on Einstein 
{\it Spin}-manifolds 
          in terms of the product formula for shifted Dirac operators. The result is based on the techniques 
          of higher variations for the Dirac operator on Einstein manifolds and spectral analysis of the 
          Dirac operator on 
          the associated Poincar\'e-Einstein metric, and relies on   
          combinatorial recurrence identities related to the dual Hahn polynomials.
\end{abstract}
\maketitle

\section{Introduction}

  Conformally covariant operators like the Yamabe-Laplace or the Dirac operator are of central 
  interest in 
  geometric analysis on manifolds. The Yamabe-Laplace operator is a representative 
  of conformally covariant operators termed GJMS operators, 
  cf.~\cite{GJMS, GZ, GoverPeterson}, 
  and this is analogous for the Dirac operator, cf.~\cite{HS,GMP1,Fischmann}. 
  Their original construction is based on the ambient metric, 
  or, equivalently, on the associated  
  Poincar\'e-Einstein metric introduced by Fefferman and Graham \cite{FG,FG3}. 
  
  In the case of GJMS operators, 
  it is shown in \cite{GoverHirachi} that in even dimensions $n$ there exists in general no conformal modification of 
  the $k$-th power of the Laplace operator for $k>\frac n2$. In the case of 
  conformal powers of the Dirac operator on general {\it Spin}-manifolds, 
  all known constructions break down for even dimensions $n$ if the order of the operator 
  exceeds the dimension. The effect 
  of non-existence for higher order conformal powers of 
  the Laplace and Dirac operators does not apply for 
  certain classes of manifolds, for example flat manifolds \cite{Slovak} or Einstein manifolds \cite{Gover}. 

  The proper understanding of 
  the internal structure 
  of conformal powers of the Laplace and Dirac operators is a difficult task, 
  see the progress for GJMS operators in \cite{Juhl1}. In particular, there exists a sequence of second order 
  differential 
  operators such that all GJMS operators are polynomials in this collection of
  second order differential 
  operators, and vice versa. Such a structure, in terms of first order differential operators, 
  is also available for low order examples 
  of conformal powers of the Dirac operator,
  cf.~\cite[Chapter~$6$]{Fischmann}.
  Explicit formulas are available on flat manifolds, 
  where they are just powers of 
  the Laplace or Dirac operators, on the spheres \cite{Branson4, ES}, where they factor into a product of 
  shifted Laplace or Dirac operators,
  or on Einstein manifolds, where the GJMS operators factor into a product of 
  shifted Laplace operators \cite{Gover, FG3}. 
  
  The main aim of our article is to complete these results for conformal powers of the Dirac operator on 
  Einstein {\it Spin}-manifolds. The result is based on 
  the proper understanding of higher variations 
  of the Dirac operator on Einstein manifolds and
  the spectral analysis of the Dirac operator on 
  the associated Poincar\'e-Einstein metric.
  The derivations of 
  specific formulas rely on combinatorial recurrence identities related to dual Hahn polynomials.

  The product structure, or factorization, 
  of conformal powers of the Laplace and Dirac operators
  is applied in theoretical physics, cf.~\cite{Dowker, Dowker1}, 
  to compute conformal and multiplicative anomalies of functional determinants in the context of the 
  $\mathrm{AdS/CFT}$ correspondence.

  \vspace{5pt}
  The paper is organized as follows. 
  In Section~\ref{Variations}, we discuss variations to all orders of
  the Dirac operator on 
  semi-Riemannian Einstein {\it Spin}-manifolds with respect to the $1$-parameter family of metrics arising from the 
  Poincar\'e-Einstein metric, cf.~Theorem~\ref{DiracVariation}. 
  In Section~\ref{HyperStuff}, we briefly recall the 
  dual Hahn polynomials, which 
  form a special class of generalized hypergeometric functions. The solution of certain recurrence relation, derived in 
  Section~\ref{ProductStructure}, has an interpretation in terms of dual Hahn polynomials. 
  Section~\ref{ProductStructure} contains 
  our main theorem, Theorem~\ref{ProductOfConformalPowers}.
  Its proof is based on 
  a recurrence relation, deduced from the construction of conformal powers on the Dirac operator of 
  a semi-Riemannian Einstein {\it Spin}-manifold via 
  the associated Poincar\'e-Einstein metric, cf.~Proposition~\ref{RecRel}. The final 
  section collects several statements and applications, aiming at the description of 
  a still largely conjectural holographic deformation of the Dirac operator.


\section{Semi-Riemannian {\it Spin}-geometry, Clifford algebras, and Poincar\'e-Einstein spaces}

In the present section we review conventions and notation related to semi-Riemannian 
{\it Spin}-geometry and 
the Poincar\'e-Einstein metric construction used throughout the article.

Let $(M,h)$ be a semi-Riemannian {\it Spin}-manifold of signature $(p,q)$ and dimension $n=p+q$. 
Then any orthonormal frame $\{e_i\}_{i=1}^n$ fulfills 
$h(e_i,e_j)=\varepsilon_i\delta_{ij}$, where $\varepsilon_i=-1$ 
for $1\leq i\leq p$ and $\varepsilon_i=1$ for $p+1\leq i\leq n$.

The Clifford algebra of $(\R^n,\langle\cdot,\cdot\rangle_{p,q})$, 
denoted by $Cl(\R^{p,q})$, is a quotient of the tensor algebra of $\R^n$ 
by 
the two sided ideal generated by 
the relations 
$x\otimes y+y\otimes x=-2\langle x,y\rangle_{p,q}$ for all $x,y\in\R^n$. 
In the even case $n=2m$, 
the complexified Clifford algebra $Cl_{\mathbb C}(\R^{p,q})$ 
has a unique irreducible representation up to isomorphism, 
whereas in the odd case $n=2m+1$ it 
has two non-equivalent irreducible representations on $\Delta_{n}:=\C^{2^m}$,
again unique up to isomorphism. 
The restriction of this representation to the spin group 
$Spin(p,q)$, regarded as a subgroup of the group of units $Cl^*(\R^{p,q})$, 
is denoted by $\kappa_{n}$. 

The choice of a {\it Spin}-structure $(Q,f)$ on $(M,h)$ 
provides an associated spinor bundle $S(M,h):=Q\times_{(Spin_0(p,q),\kappa_{n})}\Delta_{n}$, 
where $Spin_0(p,q)$ denotes the connected component of the spin group containing the identity element.
(We could work with 
the full spin group as well, because we do not need the existence of 
a scalar product). 
Then the Levi-Civita connection 
$\nabla^{h}$ on $(M,h)$ lifts to a covariant derivative 
$\nabla^{h,S}$ on the spinor bundle. The associated Dirac operator is denoted by $\slashed{D}$.

Let $\widehat{h}=e^{2\sigma}h$ be a metric conformally related to $h$, $\sigma\in \mathcal{C}^\infty(M)$. 
The spinor bundles for $\widehat{h}$ and $h$ can be identified by a vector bundle isomorphism 
$F_\sigma:S(M,h)\to S(M,\widehat{h})$, and the Dirac operator satisfies 
the conformal transformation law
\begin{align*}
  \widehat{\slashed{D}}\big(e^{\frac{1-n}{2}\sigma}\widehat{\psi}\,\big)
     =e^{-\frac{1+n}{2}\sigma}\widehat{\slashed{D}\psi},
\end{align*}
for all smooth sections $\psi\in\Gamma\big(S(M,h)\big)$, and $\widehat{\cdot}$ denotes the evaluation with 
respect to $\widehat{h}$. Conformal odd powers of the Dirac operator 
were constructed in \cite{HS, GMP1, Fischmann}, and are denoted by 
$\mathcal{D}_{2N+1}=\slashed{D}^{2N+1}+\text{LOT}$, for $N\in\N_0$ ($N<\frac n2$ for even $n$).
Here, LOT stands for ``lower order terms." They satisfy 
\begin{align*}
  \widehat{\mathcal{D}}_{2N+1}\big(e^{\frac{2N+1-n}{2}\sigma}\widehat{\psi}\,\big)
     =e^{-\frac{2N+1+n}{2}\sigma}\widehat{\mathcal{D}_{2N+1}\psi}
\end{align*}
for all smooth function $\sigma\in\mathcal{C}^\infty(M)$ and sections $\psi\in\Gamma\big(S(M,h)\big)$. 

As for the Poincar\'e-Einstein metric construction we refer to \cite{FG3}. 
The Poincar\'e-Einstein metric 
associated with an $n$-dimensional semi-Riemannian manifold $(M,h)$, 
$n\geq 3$, is $X:=M\times (0,\varepsilon)$, $\varepsilon\in\R_+$, equipped with the metric
\begin{align*}
  g_+=r^{-2}(dr^2+h_r),
\end{align*}
for a $1$-parameter family of metrics $h_r$ on $M$, $h_0=h$. The requirement of 
the Einstein condition on $g_+$ for $n$ odd, 
\begin{align*}
  Ric(g_+)+ng_+=O(r^\infty),
\end{align*}
uniquely determines the family $h_r$, while for $n$ even the conditions 
\begin{align*}
  Ric(g_+)+ng_+=O(r^{n-2}),\quad \tr(Ric(g_+)+ng_+)=O(r^{n-1}),
\end{align*}
uniquely determine the coefficients $h_{(2)},\ldots,h_{(n-2)}$, $\tilde{h}_{(n)}$ 
and the trace of $h_{(n)}$ in the formal power series
\begin{align*}
  h_r=h+r^2h_{(2)}+
\dots+r^{n-2}h_{(n-2)}+r^n(h_{(n)}+\tilde{h}_{(n)}\log r )+
\cdots . 
\end{align*}
For example, we have 
\begin{align*}
  h_{(2)}=-P,\quad h_{(4)}=\tfrac 14
\left(P^2-\tfrac{B}{n-4}\right),
\end{align*}
where $P$ is the Schouten tensor and $B$ is the Bach tensor 
associated with $h$. 

All constructions in the present article, based on the Poincar\'e-Einstein metric, 
depend for even $n$ on the coefficients $h_{(2)},\ldots,h_{(n-2)}$ and $\tr(h_{(n)})$ only. 
Choosing different representatives $h,\widehat{h}\in [h]$ in the conformal class leads to Poincar\'e-Einstein 
metrics $g^1_+$ and $g^2_+$ related by a diffeomorphism $\Phi:U_1\subset X\to U_2\subset X$, 
where both $U_i$, $i=1,2$, contain $M\times\{0\}$, $\Phi|_{M}=\id_M$, and $g_+^1=\Phi^*g^2_+$ (up to a finite 
order in $r$, for even $n$).   
     

\section{Variation of the Dirac operator induced by the Poincar\'e-Einstein metric}\label{Variations}

In this section we give a complete description 
of the variations of the Dirac operator, 
associated with the $1$-parameter family $h_r$ induced by the Poincar\'e-Einstein metric 
$g_+$, assuming that $(M,h_0=h)$ is Einstein. 

For a general $1$-parameter family of metrics $h_r$ on a Riemannian 
{\it Spin}-manifold, the first variation of the Dirac operator was 
discussed in \cite{BourguignonGauduchon}, which we will 
adapt and make explicit for the $1$-parameter family of metrics $h_r$ induced by 
the Poincar\'e-Einstein metric.

Motivated by a proof of the fundamental theorem of hypersurface theory and a new way 
to identify spinors for different metrics, \cite{BGM} introduced the technique of 
generalized cylinders to derive the first order variation formula for the Dirac 
operator with respect to a deformation of the underlying metric. 

In general, the topic of 
higher metric variations for the Dirac operator was not discussed in the literature. In case $(M,h)$ is 
Einstein, the associated Poincar\'e-Einstein metric takes a very simple form, cf.~Equation~\eqref{eq:PEMetric}. 
This allows
for a complete description of variation formulas of general order for the 
Dirac operator 
associated with $h_r$. The higher variation formulas are 
used in Section~\ref{ProductStructure}
to make the construction of conformal powers of the Dirac operator 
very explicit, ending in a product structure for $\mathcal{D}_{2N+1}$, $N\in\N_0$. 
 
Throughout the article, we use the standard notation in semi-Riemannian geometry, e.g., 
$Ric,\tau$ 
are the Ricci tensor and its scalar 
curvature, respectively. 

Let $(M,h)$ be a semi-Riemannian Einstein manifold of dimension $n$ with normalized Einstein metric $h$, 
\begin{align*}
    Ric(h)=2\lambda(n-1)h, \quad \lambda\in\R .
\end{align*}
This implies $P=\frac Jn h$, where $J=\frac{\tau}{2(n-1)}$ 
is the normalized scalar curvature and $P=\frac{1}{n-2}(Ric-Jh)$ is the Schouten tensor.
The associated Poincar\'e-Einstein metric $g_+=r^{-2}(dr^2+h_r)$ on $X$ 
is determined by the $1$-parameter family of metrics $h_r$ on $M$,
\begin{align} 
    h_r=h-r^2\tfrac Jn h+r^4\left(\tfrac{J}{2n}\right)^2h=\left(1-\tfrac{J}{2n}r^2\right)^2h,
      \quad h_0=h.\label{eq:PEMetric}
\end{align}
For $r\in{\mathbb R}_+$ small enough we consider a point-wise isomorphism 
$f_r: T_xM\to T_xM$, $x\in M$, relating $h=h_0$ and 
$h_r$ via 
\begin{align*}
    Y\mapsto f_rY:=\left(1-\tfrac{J}{2n}r^2\right)^{-1}Y,\quad
\text{for all } Y\in\Gamma(TM),
\end{align*}
characterized by $h(Y,U)=h_r(f_rY,f_rU)$ for all $Y,U\in\Gamma(TM)$ and $f_0=\id_{TM}$.
 
Let us introduce the Levi-Civita covariant derivatives on $TM$ corresponding to $h$ and $h_r$:
\begin{align*}
    \nabla^{h}, \nabla^{h_r} : \Gamma(TM)&\to\Gamma(T^*M\otimes TM),
\end{align*}
and 
\begin{align}
    \nabla^{h,h_r}:\Gamma(TM)&\to \Gamma(T^*M\otimes TM)\notag\\
     Y&\mapsto \big\{U\to \nabla^{h,h_r}_YU:=(f_r^{-1}\circ\nabla^{h_r}_{Y}\circ f_r)U\big\}.\label{eq:R-CovDer}
\end{align}
The covariant derivatives $\nabla^{h}, \nabla^{h_r}, \nabla^{h,h_r}$ extend by 
the Leibniz rule and
the spin representation to tensor-spinor fields. For example one has 
\begin{align*}
    (\nabla^{h}_U f)(Y)=\nabla^{h}_U(fY)-f(\nabla^{h}_UY),
\end{align*}   
for all $f\in End(TM)$, and $U,Y\in\Gamma(TM)$. 
\begin{lem}\label{torsion}
  The covariant derivative $\nabla^{h,h_r}$ is
  metric for $h$, and its torsion $T^r$ satisfies
  \begin{equation}
      T^r(U,Y)=f_r^{-1}\big((\nabla^{h_r}_U f_r)(Y)-(\nabla^{h_r}_Yf_r)(U)\big),\quad 
      \text{for all }U,Y\in\Gamma(TM).
  \end{equation}
\end{lem}
\begin{proof}
    Let $Y,U,Z\in\Gamma(TM)$. First we show the $h$-metricity of $\nabla^{h,h_r}$:
    \begin{align*}
      (\nabla^{h,h_r}_Yh)(U,Z)
      &=Y\big(h(U,Z)\big)-h(\nabla^{h,h_r}_YU,Z)-h(\nabla^{h,h_r}_YZ,U)\\
      &=Y\big(h_r(f_rU,f_rZ)\big)-h(f_r^{-1}\nabla^{h_r}_Y(f_rU),Z)
       -h(f_r^{-1}\nabla^{h_r}_Y(f_rZ),U)\\
      &=Y\big(h_r(f_rU,f_rZ)\big)-h_r(\nabla^{h_r}_Y(f_rU),f_rZ) 
       -h(\nabla^{h_r}_Y(f_rZ),f_rU) \\
      &=\nabla^{h_r}_Yh_r(f_rU,f_rZ)=0.
    \end{align*}
    The second statement follows from
    \begin{align*}
      T^r(U,Y)&=\nabla^{h,h_r}_UY-\nabla^{h,h_r}_YU-[U,Y]\\
      &=f_r^{-1}\big(\nabla^{h_r}_Uf_rY-\nabla^{h_r}_Yf_rU\big)-\nabla^{h_r}_UY+\nabla^{h_r}_YU\\
      &=f_r^{-1}\big(\nabla^{h_r}_Uf_rY-f_r(\nabla^{h_r}_UY)-\nabla^{h_r}_Yf_rU+f_r(\nabla^{h_r}_YU)\big)\\
      &=f_r^{-1}\big((\nabla^{h_r}_Uf_r)(Y)-(\nabla^{h_r}_Yf_r)(U)\big).
    \end{align*}
\end{proof}

It is well known that two $h$-metric covariant derivatives on $(M,h)$ differ by their torsions. 
\begin{lem}\label{ConnectionVergleich}
   We have 
   \begin{equation*}
     h(\nabla^{h,h_r}_UY,Z)=h(\nabla^{h}_UY,Z)\nonumber\\
     +\tfrac 12\big[h(T^r(Z,Y),U)-h(T^r(Y,U),Z)-h(T^r(U,Z),Y) \big]
   \end{equation*}
   for all $U,Y,Z\in\Gamma(TM)$.
\end{lem}
\begin{proof}
    Let $U,Y,Z\in\Gamma(TM)$. Any two covariant derivatives differ by a tensor field 
    $\omega\in\Gamma(T^* M\otimes TM\otimes T^* M)$, 
    i.e., $\nabla^{h,h_r}_UY=\nabla^{h}_UY+\omega(U)Y$. 
    Since both covariant derivatives are metric with respect to $h$, we get
    \begin{align}
      0&=\nabla^{h,h_r}_Uh(Y,Z)-\nabla^{h}_Uh(Y,Z)
        =-\big[h(\omega(U)Y,Z)+h(\omega(U)Z,Y)\big].\label{eq:Endo}
    \end{align} 
    Since $\nabla^h$ is 
    torsion-free, we have
    \begin{align*}
      T^r(U,Y)&=\nabla^{h,h_r}_UY-\nabla^{h,h_r}_YU-[U,Y]\\
      &=\nabla^{h,h_r}_UY-\nabla^{h,h_r}_YU-\nabla^{h}_UY+\nabla^{h}_YU\\
      &=\omega(U)Y-\omega(Y)U.
    \end{align*}
    Using Equation~\eqref{eq:Endo}, we see that this implies
    \begin{align*}
      h(T^r(Z,Y),U)&-h(T^r(Y,U),Z)-h(T^r(U,Z),Y)\\
      &=2h(\omega(U)Y,Z)=2h(\nabla^{h,h_r}_UY-\nabla^{h}_UY,Z),
    \end{align*}
    and the proof is complete.
\end{proof}

Any $h$-metric covariant derivative induces a covariant derivative on the 
spinor bundle $S(M,h)$, hence $\nabla^h$, $\nabla^{h_r}$ and $\nabla^{h,h_r}$ induce
\begin{align}
    \nabla^{h,S}:\quad &\Gamma\big(S(M,h)\big)\to\Gamma\big(T^*M\otimes S(M,h)\big),\nonumber \\
    \nabla^{h_r,S}:\quad &\Gamma\big(S(M,h_r)\big)\to\Gamma\big(T^*M\otimes S(M,h_r)\big),\nonumber \\
    \nabla^{r,S}:\quad &\Gamma\big(S(M,h)\big)\to\Gamma\big(T^*M\otimes S(M,h)\big).
\end{align}
Note that the 
last of the above covariant derivatives equals 
$\nabla^{h,h_r,S}$, but we will 
use the abbreviation $\nabla^{r,S}$. It follows from Lemma~\ref{ConnectionVergleich} that locally
\begin{align}
    \nabla^{r,S}_{s_i}\psi&=s_i(\psi)+\tfrac 12\sum_{j < k}\varepsilon_j\varepsilon_k
        h(\nabla^{h,h_r}_{s_i}s_j,s_k)s_j\cdot s_k\cdot\psi\notag\\
    &=\nabla^{h,S}_{s_i}\psi+\tfrac 18\sum_{j\neq k}\varepsilon_j\varepsilon_k 
        T^{r,\sigma}_{ijk}s_j\cdot s_k\cdot\psi,\label{eq:SpinorVergleich}
\end{align}
where $\psi\in\Gamma\big(S(M,h)\big)$, $\{s_i\}_{i=1}^n$ is an $h$-orthonormal 
frame, $T^{r,\sigma}_{ijk}:=(1-\sigma-\sigma^2)T^r_{kji}$ with $\sigma T^r_{ijk}:=T^r_{jki}$ and 
$T^r(U,Y,Z):=h(T^r(U,Y),Z)$ for $U,Y,Z\in\Gamma(TM)$. 
The isometry $f_r:T_xM\to T_xM$, $x\in M$, pulls-back $h$ to $h_r$ and lifts to a spinor bundle isomorphism 
\begin{align*}
   \beta_r: S(M,h)\to S(M,h_r).
\end{align*}
It preserves the base points on $M$ and satisfies 
$\beta_r(U\cdot\psi)=f_r(U)\cdot\beta_r(\psi)$ for all $U\in\Gamma(TM),\psi\in\Gamma\big(S(M,h)\big)$. 
\begin{lem}\label{SpinCovariantDerivatives}
    Let $\psi\in\Gamma\big(S(M,h)\big), Y\in\Gamma(TM)$. Then 
   \begin{align}
      \nabla^{r,S}_Y\psi=(\beta_r^{-1}\circ\nabla^{h_r,S}_Y\circ\beta_r)\psi.
   \end{align}
\end{lem}
\begin{proof}
  For $\psi\in\Gamma\big(S(M,h)\big)$, we have 
  \begin{align*}
      \nabla^{r,S}_{Y}\psi&=Y(\psi)+\tfrac 12\sum_{j<k}\varepsilon_j\varepsilon_k
          h(\nabla^{h,h_r}_Ys_j,s_k)s_j\cdot s_k\cdot\psi\\
      &=Y(\psi)+\tfrac 12\sum_{j<k}\varepsilon_j\varepsilon_k
          h_r(\nabla^{h_r}_Yf_rs_j,f_rs_k)\beta^{-1}_r\big(f_rs_j\cdot f_rs_k\cdot\beta_r\psi\big)\\
      &= (\beta_r^{-1}\circ\nabla^{h_r,S}_Y\circ\beta_r)\psi,
  \end{align*}
  which completes the proof.
\end{proof}

Let us introduce the notation 
\begin{equation}
    f^{(l)}:=\tfrac{d^l}{dr^l}(f_r)|_{r=0},\quad T^{(l)}(U,Y):=\tfrac{d^l}{dr^l}\big(T^r(U,Y)\big)|_{r=0},
\end{equation}
for $l\in\N_0$. 
\begin{lem}\label{Ableitungen}
    \begin{enumerate}
       \item Let $l\in\N_0$. Then
         \begin{equation}
           f^{(2l+1)}U=0,\quad f^{(2l)}U=\tfrac{(2l)!}{2^l}\left(\tfrac Jn\right)^lU,\quad U\in\Gamma(TM).
         \end{equation}
       \item  
       Let $l\in\N$. Then the torsion $T^r$ of $\nabla^{h,h_r}$ fulfills
       \begin{align} 
          T^{(l)}(U,Y)=0
       \end{align}
       for all $U,Y\in\Gamma(TM)$.
    \end{enumerate}
\end{lem}
\begin{proof}
  Expansion of $f_r$ into a formal power series 
  \begin{align*}
      f_r=\left(1-\tfrac{J}{2n}r^2\right)^{-1}\id_{TM}
        =\sum_{l\geq 0} \tfrac{r^{2l}}{(2l)!} (2l)!\left(\tfrac{J}{2n}\right)^l\id_{TM}
  \end{align*}
  gives the first statements. It follows from Lemma~\ref{torsion} that the $l$-th derivative 
  of $T^r$ at $r=0$ is given by a sum, where each contribution 
  contains derivatives of $f_r$, $f_r^{-1}$ and $\nabla^{h_r}$ at $r=0$. 
  Using $f^{(2l+1)}=0$ and 
  $f^{(2l)}=\frac{(2l)!}{2^l}(\frac Jn)^l \id_{TM}$, for all $l\in \N_0$, we just have to show that  
  $\frac{d^l}{dr^l}(\nabla^{h_r})|_{r=0}$ acts as a covariant derivative, hence annihilating the identity map. 
  But this is obvious, since $\frac{d^l}{dr^l}$ 
  does not effect 
the property of $\nabla^{h_r}$ being a covariant derivative.
\end{proof}

In what follows, we 
use two Dirac operators induced by $\nabla^{h,S}$ and $\nabla^{h_r,S}$:
\begin{align}
    \slashed{D}^h:\quad & \Gamma\big(S(M,h)\big)\to\Gamma\big(S(M,h)\big),\nonumber\\
    \slashed{D}^{h_r}:\quad & \Gamma\big(S(M,h_r)\big)\to\Gamma\big(S(M,h_r)\big).
\end{align}
Furthermore, 
we define
\begin{align}
    \slashed{D}^{h,h_r}: \Gamma\big(S(M,h)\big)&\to\Gamma\big(S(M,h)\big)\nonumber\\
    \psi&\mapsto \beta_r^{-1}\circ \slashed{D}^{h_r}\circ\beta_r(\psi).
\end{align}
Lemma~\ref{SpinCovariantDerivatives} and $\beta_r(U\cdot\psi)=f_r(U)\cdot\beta_r(\psi)$ for 
$U\in\Gamma(TM),\psi\in\Gamma\big(S(M,h)\big)$ imply
\begin{equation}
    \slashed{D}^{h,h_r}\psi=\sum_{i=1}^n\varepsilon_i s_i\cdot\nabla^{r,S}_{f_r(s_i)}\psi ,
\end{equation}
and the $r$-derivatives of 
$\slashed{D}^{h,h_r}$ at $r=0$ yield the variation formulas 
for $\slashed{D}^h$ with respect to the $1$-parameter deformation $h_r$ of $h$. Note that 
$\slashed{D}^{h,h_r}$ is not the Dirac operator induced by $\nabla^{r,S}$.

\begin{theorem}\label{DiracVariation}
  Let $(M,h)$
  be a semi-Riemannian Einstein {\it Spin}-manifold 
  with\break $\dim(M)=n$, i.e., $Ric=2(n-1)\lambda h$.
  Let $g_+$ 
  be the associated Poincar\'e-Einstein metric
  on $X$ with $g_+=r^{-2}(dr^2+h_r)$, $h_r=(1-\frac{J}{2n}r^2)^2h$. 
  Then, for all $l\in \N_0, \psi\in\Gamma\big(S(M,h)\big)$, we have 
  \begin{align}
      \tfrac{d^{(2l)}}{dr^{(2l)}}(\slashed{D}^{h,h_r}\psi)|_{r=0}
        &= (2l)!\left(\tfrac{J}{2n}\right)^l\slashed{D}^h\psi,\nonumber\\
      \tfrac{d^{(2l+1)}}{dr^{(2l+1)}}(\slashed{D}^{h,h_r}\psi)|_{r=0}&= 0.
  \end{align} 
\end{theorem}
\begin{proof}
     The $r$-derivatives of $\slashed{D}^{h,h_r}$ at $r=0$ are
     \begin{equation*}
        \tfrac{d^l}{dr^l}(\slashed{D}^{h,h_r}\psi)|_{r=0}=\sum_{k=0}^l\binom{l}{k}
           \sum_{i=1}^n\varepsilon_i s_i\cdot\nabla^{r,S;(k)}_{f^{(l-k)}s_i}\psi,
      \end{equation*}
      where $\nabla^{r,S;(k)}:=\frac{d^k}{dr^k}(\nabla^{r,S})|_{r=0}$, for $k\in\N_0$. 
      From Equation~\eqref{eq:SpinorVergleich}, we obtain
      \begin{align*}
        \tfrac{d^l}{dr^l}(\nabla^{r,S}_{s_i}\psi)
          =\tfrac 18\sum_{j\neq k}\varepsilon_j\varepsilon_k\frac{d^l}{dr^l}(T^{r,\sigma}_{ijk}) s_j\cdot s_k\cdot\psi ,
      \end{align*}
      which vanishes at $r=0$ due to Lemma~\ref{Ableitungen} and 
      the linearity of $\frac{d^l}{dr^l}$. Thus we get 
      \begin{align*}
         \tfrac{d^l}{dr^l}(\slashed{D}^{h,h_r}\psi)|_{r=0}
           =\sum_{i=1}^n\varepsilon_i s_i\cdot \nabla^{0,S}_{f^{(l)}s_i}\psi.
      \end{align*}
      Since $\nabla^{0,S}$ agrees with the spinor covariant derivative $\nabla^{h,S}$ on 
      $S(M,h)$, we may 
      conclude by Lemma~\ref{Ableitungen} that 
      \begin{align*}
        \tfrac{d^{2l}}{dr^{2l}}(\slashed{D}^{h,h_r}\psi)|_{r=0}
          =(2l)!\left(\tfrac{J}{2n}\right)^l\slashed{D}^h\psi, 
          \quad \tfrac{d^{2l+1}}{dr^{2l+1}}(\slashed{D}^{h,h_r}\psi)|_{r=0}=0,
      \end{align*}
      hence completing the proof. 
\end{proof}


\section{Generalized hypergeometric functions and dual Hahn polynomials}\label{HyperStuff}

The aim of the present section is to introduce 
a certain class of polynomials, to prove some of their combinatorial 
properties, and to give their interpretation 
in terms of dual Hahn 
polynomials. These polynomials will be responsible for the product 
structure of conformal powers of the Dirac operator.

The Pochhammer symbol of a complex number $a\in\C$ is denoted 
by
$(a)_l$, and it is defined by $(a)_l:=a(a+1)
\cdots(a+l-1)$ for $l\in\N$, and $(a)_0:=1$. 
The generalized hypergeometric function ${}_pF_q$, for $p,q\in\N$,
with $p$ upper parameters, $q$ lower parameters, and argument $z$, is defined by
\begin{align}
    {}_pF_q
  \left[\begin{matrix}a_1\;,\;\ldots\;,\;a_p\\b_1\;,\;\ldots\;,\;b_q\end{matrix};z\right]
      :=\sum_{l=0}^\infty\frac{(a_1)_l
  \cdots (a_p)_l}{(b_1)_l
  \cdots (b_q)_l}\frac{z^l}{l!},
\end{align}
for $a_i\in\C$ ($1\leq i\leq q$), $b_j\in\C\setminus\{-\N_0\}$ ($1\leq j\leq q$), and $z\in\C$. 

For later purposes, we introduce the polynomials 
\begin{align}
  \tilde{q}_m(y):=\sum_{l=0}^m (-1)^{m-l}(
  \tfrac n2+1+l)_{m-l}\,(k-m)_{m-l}
  \binom ml \prod_{j=1}^l(y-j^2), \label{eq:SolutionSpinor}
\end{align}
for $m\in\N_0$, $k,n\in\N$ and an abstract variable $y$.   
\begin{sat}\label{RecRel}
  The polynomials $\tilde{q}_m(y)$, $m\in\N_0$, satisfy the recurrence 
  relation
  \begin{multline} \label{eq:rec}
      \tilde{q}_{m+1}(y)=\left(y-2m(k-m-\tfrac n2-\tfrac 12)-\tfrac n2(k-1)-k\right)\tilde{q}_m(y)\\
      -m(m-k)(m+\tfrac n2)(m-k+\tfrac n2-1)\tilde{q}_{m-1}(y),
    \end{multline}
  with $\tilde{q}_{-1}(y):=0$ and $\tilde{q}_0(y):=1$.
\end{sat}
\begin{proof}
We prove the 
statement by comparing the coefficients of $\prod_{j=1}^l(y-j^2)$
on both sides of \eqref{eq:rec}. 
The only detail to observe is that one must replace $y$ in the
coefficient of $\tilde{q}_m(y)$ on the right-hand side of \eqref{eq:rec}
by $(y-(l+1)^2)+(l+1)^2$, with $l$ being the summation index of
the sum in the definition of $\tilde{q}_m(y)$. The term $(y-(l+1)^2)$
then combines with the product $\prod_{j=1}^l(y-j^2)$ to become
$\prod_{j=1}^{l+1}(y-j^2)$.
The verification that the coefficients of $\prod_{j=1}^l(y-j^2)$
do indeed agree is then a routine matter.
\end{proof}

\begin{bem}
  Our considerations are motivated by \cite{FG3},
  where the analogous recurrence relation 
  \begin{multline*}
      q_{m+1}(y)=\left(y-2m(k-m-\tfrac n2)-\tfrac n2(k-1)\right)q_m(y)\\
      -m(m-k)(m-1+\tfrac n2)(m-1+\tfrac n2-k)q_{m-1}(y),
  \end{multline*}
  for $k,n\in\N$, $m\in\N_0$, and $q_{-1}(y):=0$, $q_0(y):=1$, appears. 
  Its solution is given by 
  \begin{align}
      q_m(y):=\sum_{l=0}^m (-1)^{m-l}(\tfrac n2+l)_{m-l}\,(k-m)_{m-l}
      \binom ml \prod_{j=1}^l\big(y-j(j-1)\big).\label{eq:SolutionFunction}
  \end{align}
\end{bem}

In the rest of the 
section, we discuss 
interpretations of $\tilde{q}_m(y)$ and $q_m(y)$ 
in terms of dual Hahn polynomials, cf.~
\cite{KMcG,KLS}. The Hahn polynomial $Q_n(x):=Q_n(x;\alpha,\beta,\break N)$ 
is defined by 
\begin{align}
    Q_n(x):={}_3F_2\left[\begin{matrix}-n\;,\;-x\;,\;n+\alpha+\beta+1\\ \alpha+1\;,\;-N+1\end{matrix}; 1\right],
\end{align}
for 
$\Re(\alpha), \Re(\beta)>-1$, $N\in\N$ and $n=0,\ldots,N-1$. It is known 
that, beside recurrence relations, Hahn polynomials satisfy
a difference relation, cf.~\cite[Equation~(1.3)]{KMcG}. The dual Hahn polynomials 
can be defined by recurrence relations with the same coefficients 
as the Hahn polynomials have in their difference relations, cf.~\cite[Equation~(1.18)]{KMcG}.

For $\lambda(n):=n(n+\alpha+\beta+1)$, 
the relation between Hahn polynomials $Q_n(x)$ and dual Hahn polynomials 
$R_k(\lambda):=R_k(\lambda;\alpha,\beta,N)$ is given by
\begin{align}
    R_k\big(\lambda(n)\big)=Q_n(k). 
\end{align}

Notice that
\begin{align*}
  (-y)_l\,(1+y)_l&=(-y)(y+1)
  \cdots(-y+j)
  \cdot(y+1+j)
  \cdots(-y+l-1)(y+l)\\
  &=(-1)^l\prod_{j=1}^{l}\big(y(y+1)-j(j-1)\big),\\
  (1-y)_l\,(y+1)_l&=(-1)^l\prod_{j=1}^{l}\big(y^2-j^2\big),
\end{align*}
for $l\in\N$. 
Furthermore, by using the identities 
for Pochhammer symbols
\begin{align*}
    (\tfrac n2+l)_{m-l}&=(\tfrac n2)_m\frac{1}{(\tfrac n2)_l},
      \quad  (\tfrac n2+1+l)_{m-l}=(\tfrac n2+1)_m\frac{1}{(\tfrac n2+1)_l},\\
    (k-m)_{m-l}&=(k-m)_m\frac{(-1)^l}{(1-k)_l},\quad 
    \binom ml  =(-1)^l\frac{(-m)_l}{l!},
\end{align*}
one obtains the following precise relations.
\begin{sat}
  For all $m\in\N_0, k,n\in\N$, we have
  \begin{align}
    \tilde{q}_m\big(\lambda(y-1)\big)&=(-1)^m(\tfrac n2+1)_m(k-m)_m
    ~{}_3F_2\left[\begin{matrix}-(y-1)\;,\;-m\;,\;1+y\\ \tfrac n2+1\;,\;1-k\end{matrix}; 1\right]\nonumber\\
    &=(-1)^m(\tfrac n2+1)_m(k-m)_m~ Q_{y-1}(m;\tfrac{n}{2},1-\tfrac{n}{2},k)\nonumber\\
    &=(-1)^m(\tfrac n2+1)_m(k-m)_m~ R_m\big(\lambda(y-1);\tfrac n2,1-\tfrac n2,k\big)
  \end{align}
  and
  \begin{align}
    q_m\big(\lambda(y)\big)&=(\tfrac n2)_m(k-m)_m(-1)^m
    ~{}_3F_2\left[\begin{matrix}-y\;,\;-m\;,\;1+y\\ \tfrac n2\;,\;1-k\end{matrix}; 1\right]\nonumber\\
    &=(\tfrac n2)_m(k-m)_m(-1)^m~ Q_{y}(m;\tfrac{n}{2}-1,1-\tfrac{n}{2},k)\nonumber\\
    &=(\tfrac n2)_m(k-m)_m(-1)^m~ R_m\big(\lambda(y);\tfrac n2-1,1-\tfrac n2,k\big).
  \end{align}
\end{sat}  
Hence, up to a 
multiplicative factor, both $\tilde{q}_m(y)$ and $q_m(y)$ can be realized as dual Hahn polynomials. 


\section{Product structure (factorization) of conformal powers of the Dirac operator}\label{ProductStructure}
  
In the present 
section, we show that conformal powers of the Dirac operator on Einstein manifolds 
obey a 
product structure,
in the sense that they factor into linear factors based on shifted Dirac operators. 
This result is parallel to the case of conformal 
powers of the Laplace operator on Einstein manifolds, cf.~\cite[Theorem~$1.2$]{Gover}. 
  
Let us denote the Dirac operator on $(M,h)$ by $\slashed{D}$. 
(Notice that in Section~\ref{Variations} we used $\slashed{D}^h$ instead of $\slashed{D}$.)
The proof of 
our main result, Theorem~\ref{ProductOfConformalPowers},
relies on the construction of 
conformal powers of the Dirac operator. 
\begin{theorem}[\sc \cite{GMP1}]\label{existence}
    Let $(M,h)$ be a semi-Riemannian {\it Spin}-manifold of dimension $n$. For every $N\in\N_0$ ($N\leq \tfrac n2$ 
    for even $n$) there exists a linear differential operator, called conformal 
    power of the Dirac operator, 
    \begin{align}
      \mathcal{D}_{2N+1}:\Gamma\big(S(M,h)\big)\to\Gamma\big(S(M,h)\big),
    \end{align}
    satisfying
    \begin{enumerate}
      \item $\mathcal{D}_{2N+1}$ is of order $2N+1$ 
        and $\mathcal{D}_{2N+1}=\slashed{D}^{2N+1}+\text{LOT}$, where,
        as before, LOT denotes lower order terms;
      \item $\mathcal{D}_{2N+1}$ is conformally 
        covariant, that is, 
        \begin{align}
           \widehat{\mathcal{D}}_{2N+1}\left(e^{\frac{2N+1-n}{2}\sigma}\widehat{\psi}\right)
             =e^{-\frac{2N+1+n}{2}\sigma}\widehat{\mathcal{D}_{2N+1}\psi}
        \end{align}
        for every $\psi\in\Gamma\big(S(M,h)\big)$, $\sigma\in\mathcal{C}^\infty(M)$. 
    \end{enumerate}
\end{theorem}
We briefly outline the main point of the proof, which will be 
then analyzed in detail on Einstein manifolds. 
Let $g_+$ 
be the associated Poincar\'e-Einstein metric
on $X$ with conformal infinity $(M,[h])$. The conformal compactification of 
$\big(X=M\times (0,\varepsilon),g_+\big)$ is 
\begin{align*}
    \big(M\times [0,\varepsilon),\bar{g}:=r^2g_+=dr^2+h_r),
\end{align*}
where $\bar{g}$ smoothly extends to $r=0$. Corresponding spinor bundles are denoted by 
\begin{align*}
    S(M,h),\quad S(X,g_+),\quad S(X,\bar{g}),
\end{align*}
respectively. The spinor bundle $S(X,\bar{g})|_{r=0}$ is 
isomorphic to $S(M,h)$
if $n$ is even, 
and it is isomorphic to $S(M,h)\oplus S(M,h)$ 
if $n$ is odd. 
The proof of Theorem~\ref{existence} is based on the extension of a boundary 
spinor $\psi\in\Gamma(S(X,\bar{g})|_{r=0})$ to the 
interior $\theta\in\Gamma\big(S(X,\bar{g})\big)$: one requires $\theta$ to be a formal solution of 
\begin{align}
    D(\bar{g})\theta=i\lambda\theta,\quad \lambda\in\C .\label{eq:EigenEquation}
\end{align}
Here, $D(\bar{g})$ arises by applying the vector bundle isomorphism 
$F_r:S(X,g_+)\to S(X,\bar{g})$, which exists since $g_+$ and $\bar{g}$ are conformally equivalent, to 
the equation $\slashed{D}^{g_+}\varphi=i\lambda\varphi$, $\lambda\in\C$ 
and $\varphi\in\Gamma\big(S(X,g_+)\big)$. 
The solution of Equation~\eqref{eq:EigenEquation} is obstructed for $\lambda=-\frac{2N+1}{2}$, 
and the obstruction induces a conformally covariant 
linear differential operator $\mathcal{D}_{2N+1}=\slashed{D}^{2N+1}+LOT$.  
 
Let us be more specific. Let $(M,h)$ be 
a semi-Riemannian Einstein {\it Spin}-manifold, normalized by 
$Ric(h)=\frac{2(n-1)J}{n}h$ for constant normalized scalar 
curvature $J\in\R$. Consider the embedding $\iota_r:M\to X$ given by 
$\iota_r(m):=(r,m)$. Then $(M,\iota_r^*({\bar g})=h_r)$  
is a hypersurface in $(X,\bar{g})$ with trivial space-like normal bundle. 
It follows from \cite{BGM} that the Dirac operator $\slashed{D}^{\bar{g}}$ of $(X,\bar{g})$ and 
the leaf-wise (or, hypersurface) Dirac operator 
\begin{align*}
    \widetilde{\slashed{D}}^{h_r}:=\partial_r\cdot
       \sum_{i=1}^n\varepsilon_is_i\cdot\widetilde{\nabla}^{h_r,S}_{s_i}
      :\Gamma\big(S(X,\bar{g})\big)\to\Gamma\big(S(X,\bar{g})\big)
\end{align*}
for an $h_r$-orthonormal frame $\{s_i\}_i$ on $M$ are related by 
\begin{align}\label{hyperdirac}
    \iota_r^*\partial_r\cdot\slashed{D}^{\bar{g}}=\widetilde{\slashed{D}}^{h_r}\iota_r^*
       +\tfrac n2\iota_r^* H_r-\iota_r^*\nabla^{\bar{g},S}_{\partial_r},
\end{align}
where $H_r:=\frac 1n\tr_{h_r}(W_r)$ is the $h_r$-trace of the Weingarten map 
associated with the 
embedding $\iota_r$. 
We used a swung dash (on $\widetilde{\slashed{D}}^{h_r}$) in order to 
emphasize the action on the spinor 
bundle on $(X,\bar{g})$. At $r=0$, we have the identification 
\begin{align*}
    \widetilde{\slashed{D}}:=\widetilde{\slashed{D}}^{h_0}\simeq
     \begin{cases} 
          \slashed{D},
          & \text{if $n$ is even},\\
          \begin{pmatrix}\slashed{D}&0\\0&-\slashed{D}\end{pmatrix},
          & \text{if $n$ is odd.}
     \end{cases}
\end{align*} 
The equation $\slashed{D}^{g_+}\varphi=i\lambda\varphi$, $\lambda\in\C$ 
and $\varphi\in\Gamma\big(S(X,g_+)\big)$, is equivalent to 
Equation~\eqref{eq:EigenEquation} by combination of conformal 
covariance, Equation~\eqref{hyperdirac}, and the isomorphism 
$F_r$, where the linear differential operator 
$D(\bar{g}):\Gamma\big(S(X,\bar{g})\big)\to \Gamma\big(S(X,\bar{g})\big)$ is given by 
\begin{align*}
    D(\bar{g})\theta=-r\partial_r\cdot\widetilde{\slashed{D}}^{h_r}\theta-\tfrac n2 r H_r\partial_r\cdot\theta
        + r\partial_r\cdot\nabla^{\bar{g},S}_{\partial_r}\theta-\tfrac n2 \partial_r\cdot\theta
\end{align*}
for $\theta=F_r(\varphi)$. Using Theorem~\ref{DiracVariation}, we find 
the explicit formulas 
\begin{align}\label{Dhhrcomp}
    \widetilde{\slashed{D}}^{h_r}=\left(1-\tfrac{J}{2n}r^2\right)^{-1}\widetilde{\slashed{D}},
       \quad H_r=\tfrac Jn r\left(1-\tfrac{J}{2n}r^2\right)^{-1}.
\end{align}
This is a consequence of the Einstein assumption on $M$.
In general, there is no explicit formula 
analogous to Equation \eqref{Dhhrcomp}. 
We decompose the spinor 
bundle $S(X,\bar{g})$ into the $\pm i$-eigenspaces $S^{\pm \partial_r}(X,\bar{g})$ 
with respect to the linear map $\partial_r\cdot : S(X,\bar{g})\to S(X,\bar{g})$ 
satisfying $\partial_r^2=-1$. 
The formal solution of Equation~\eqref{eq:EigenEquation} is constructed inside
\begin{align*}
    \mathcal{A}:=\big\{\theta=\sum_{j\geq 0}r^j\theta_j~|~\theta_j\in\Gamma\big(S(X,\bar{g})\big),\;
       \nabla^{\bar{g},S}_{\partial_r}\theta_j=0\big\},
\end{align*}
and
\begin{align*}
   \bar{\theta}:=r^{\frac n2+\lambda}\theta=\sum_{j\geq 0}r^{\frac n2+\lambda+j}
     (\theta^+_j+\theta^-_j)\in\mathcal{A}
\end{align*}
for $\theta_j^\pm\in\Gamma\big(S^{\pm\partial_r}(X,\bar{g})\big), j\in\N_0$, 
is a solution of Equation~\eqref{eq:EigenEquation} provided the coupled system
of recurrence relations
\begin{align}
    j\theta^+_j&=\widetilde{\slashed{D}}\theta^-_{j-1}+\tfrac{n+j-2}{2n}J\theta^+_{j-2},\notag\\
    (2\lambda+j)\theta_j^- &=\widetilde{\slashed{D}}\theta^+_{j-1}
      +\tfrac{2\lambda+n+j-2}{2n}J\theta^-_{j-2},\label{eq:System}
\end{align}
holds for all $j\in\N_0$. Note that we only consider restrictions to $r=0$ and then 
extend $\theta^\pm_j$, $j\geq 0$, by parallel transport with respect to $\nabla^{\bar{g},S}$ 
along the geodesic induced by the $r$-coordinate. The initial data are given by 
$\theta_0^+:=\psi^+$ for some 
$\psi^+\in\Gamma(S^{+\partial_r}(X,\bar{g})|_{r=0})$, and $\theta_0^-=0$. 
Assuming $\lambda\notin -\N+\frac 12$, the system can be solved uniquely for all 
$j\in\N$ if $n$ is odd, and for all $j\in\N$ such that $j\leq n$ if $n$ is even. 
The obstruction at $\lambda=-\frac{2N+1}{2}$, for $N\in\N_0$ ($N\leq \frac n2$ for even $n$),  
is given by $\mathcal{D}_{2N+1}$ for $N\in\N$.

The application of $\widetilde{\slashed{D}}$ to the system \eqref{eq:System}
together with the shift of $j$ to $j-1$, 
respectively $j-3$, implies
\begin{align*}
      (j-1)\widetilde{\slashed{D}}\theta^+_{j-1}&=\widetilde{\slashed{D}}^2\theta_{j-2}^-
         +\tfrac{n+j-3}{2n}J\widetilde{\slashed{D}}\theta^+_{j-3},\\
      (2\lambda+j-1)\widetilde{\slashed{D}}\theta^-_{j-1}&=\widetilde{\slashed{D}}^2\theta^+_{j-2}
         +\tfrac{2\lambda+n+j-3}{2n}J\widetilde{\slashed{D}}\theta^-_{j-3},\\
      \widetilde{\slashed{D}}\theta^-_{j-3}&= (j-2)\theta^+_{j-2}-\tfrac{n+j-4}{2n}J\theta^+_{j-4},\\
      \widetilde{\slashed{D}}\theta^+_{j-3}&= (2\lambda+j-2)\theta^-_{j-2}
         -\tfrac{2\lambda+n+j-4}{2n}J\theta^-_{j-4}.
\end{align*}
These formulas can be used to decouple the system \eqref{eq:System} into 
\begin{align}
      j\theta^+_j=&\left(\tfrac{1}{2\lambda+j-1}\widetilde{\slashed{D}}^2
        +\tfrac{(2\lambda+n+j-3)(j-2)}{2n(2\lambda+j-1)}J
        +\tfrac{(2\lambda+j-1)(n+j-2)}{2n(2\lambda+j-1)}J\right)\theta^+_{j-2}\notag\\
     &-\tfrac{(2\lambda+n+j-3)(n+j-4)}{4n^2(2\lambda+j-1)}J^2\theta^+_{j-4},\notag\\
     (2\lambda+j)\theta^-_j=&\left(\tfrac{1}{j-1}\widetilde{\slashed{D}}^2
        +\tfrac{(2\lambda+j-2)(n+j-3)}{2n(j-1)}J
        +\tfrac{(2\lambda+n+j-2)(j-1)}{2n(j-1)}J\right)\theta^-_{j-2}\notag\\
     &-\tfrac{(2\lambda+n+j-4)(n+j-3)}{4n^2(j-1)}J^2\theta^-_{j-4},\label{eq:DecoupledSystem}
\end{align}
for all $j\geq 2$, with the initial data $\theta^+_0=\psi^+$, $\theta^-_0=0$, $\theta^+_1=0$, 
and $\theta^-_1=\frac{1}{2\lambda+1}\widetilde{\slashed{D}}\psi^+$. 
Introducing $\phi_l:=\theta^-_{2l+1}$ for $l\in\N_0$, Equation~\eqref{eq:DecoupledSystem} is 
equivalent to 
\begin{align}
      2l(2\lambda+2l+1)\phi_l=&\left(\widetilde{\slashed{D}}^2
       +\tfrac{(2\lambda+2l-1)(n+2l-2)}{2n}J
       +\tfrac{2l(2\lambda+n+2l-1)}{2n}J\right)\phi_{l-1},\notag\\
      &-\tfrac{(2\lambda+n+2l-3)(n+2l-2)}{4n^2}J^2\phi_{l-2},  \label{eq:recurence}
\end{align}
for $l\in\N$ and $\phi_0:=\frac{1}{2\lambda+1}\widetilde{\slashed{D}}\psi^+$. We define 
the solution operators $\tilde{q}_l(y)$ by
\begin{align}
      4^ll!\,(\tfrac{n}{2J})^l(\lambda+\tfrac 32)_l\,\phi_l=\tilde{q}_l(y)\phi_0,\label{eq:Obstruction}
\end{align}
where $y:=\frac{n}{2J}\widetilde{\slashed{D}}^2$. Then 
Equation~\eqref{eq:recurence} yields a recurrence relation for $\tilde{q}_l(y)$,
namely
\begin{multline*}
      \tilde{q}_l(y)=\big(y+(\lambda+l-\tfrac 12)(l+\tfrac n2-1)
         +l(\lambda+l+\tfrac n2-\tfrac 12)\big)\tilde{q}_{l-1}(y)\\
      -(l-1)(l+\tfrac n2-1)(l+\lambda-\tfrac 12)(l+\lambda+\tfrac n2-\tfrac 32)\tilde{q}_{l-2}(y),
\end{multline*}
$l\in\N$, $\tilde{q}_{-1}(y):=0$ and $\tilde{q}_0(y):=1$. Changing $l$ to $(m+1)$ and 
substituting $\lambda=-\frac{2N+1}{2}$ for $N\in\N_0$,
we obtain
\begin{multline}
      \tilde{q}_{m+1}(y)=\big(y-2m(N-m-\tfrac n2-\tfrac 12)-\tfrac n2(N-1)-N\big)\tilde{q}_{m}(y)\notag\\
      -m(m-N)(m+\tfrac n2)(m-N+\tfrac n2-1)\tilde{q}_{m-1}(y).\label{eq:recurence1}
\end{multline}
The unique solution of the recurrence relation \eqref{eq:recurence1} is given by 
\begin{align}
      \tilde{q}_m(y):=\sum_{l=0}^m (-1)^{m-l}(\tfrac n2+1+l)_{m-l}\,(N-m)_{m-l}
        \binom ml  \prod_{j=1}^l(y-j^2),
\end{align}
cf.~Proposition~\ref{RecRel}, and it specializes for $m=N$ to 
\begin{align*}
      \tilde{q}_N(y)=\prod_{j=1}^N(y-j^2)
        =\prod_{j=1}^N\left(\sqrt{\tfrac{n}{2J}}\widetilde{\slashed{D}}-j\right)
            \left(\sqrt{\tfrac{n}{2J}}\widetilde{\slashed{D}}+j\right). 
\end{align*}
The solution $\phi_l$, cf.~Equation~\eqref{eq:Obstruction}, multiplied by $(2\lambda+1)$ 
is obstructed at $\lambda=-\frac{2N+1}{2}$, $N\in\N_0$, and we get
\begin{align*}
    4^NN!(-N)_N\phi_N=
    \left(\tfrac{n}{2J}\right)^{-N}\tilde{q}_N(y)\widetilde{\slashed{D}}\psi^+.
\end{align*}
Repeating all the previous steps with eigen-equation~\eqref{eq:EigenEquation} for the eigenvalue 
$-\lambda$ and initial data $\psi^-\in\Gamma(S^{-\partial_r}(X,\bar{g})|_{r=0})$, the obstruction 
at $\lambda=-\frac{2N+1}{2}$, for $N\in\N_0$, induces 
\begin{align}
   \mathcal{D}_{2N+1}&=(\frac{n}{2J})^{-N}\slashed{D}
     \prod_{j=1}^N\left(\sqrt{\tfrac{n}{2J}}\slashed{D}-j\right)
     \left(\sqrt{\tfrac{n}{2J}}\slashed{D}+j\right)\nonumber\\
   &=\slashed{D}\prod_{j=1}^N\left(\slashed{D}-j\sqrt{\tfrac{2J}{n}}\right)\left(\slashed{D}
      +j\sqrt{\tfrac{2J}{n}}\right),
\end{align}
the conformal power of the Dirac operator in the factorized form. 
Note that there is no restriction on $N\in\N_0$ in the case of even $n$. 
Thus we have
the following result.
\begin{theorem}\label{ProductOfConformalPowers}
    Let $(M,h)$ be a semi-Riemannian Einstein {\it Spin}-manifold of dimension $n$, normalized 
    by $Ric(h)=\frac{2(n-1)J}{n}h$ for constant normalized scalar curvature $J\in\R$.

    The $(2N+1)$-th conformal power of the Dirac operator, $N\in\N_0$, satisfies
    \begin{align}
      \mathcal{D}_{2N+1}\psi&=\prod_{j=1}^{2N+1}
         \left(\slashed{D}-(N-j+1)\sqrt{\tfrac{2J}{n}}\right)\psi\nonumber\\
      &=\slashed{D}\prod_{j=1}^N\left(\slashed{D}^2-j^2\left(\tfrac{2J}{n}\right)\right)\psi
    \end{align}
    for all $\psi\in\Gamma\big(S(M,h)\big)$. The empty product is regarded as $1$. 
\end{theorem}

In particular, Theorem~\ref{ProductOfConformalPowers} applies to the standard round sphere $(S^n,h)$ 
of radius $1$. We get
\begin{align}
    \mathcal{D}_{2N+1}=(\slashed{D}-N)
    \cdots(\slashed{D}-1)\slashed{D}(\slashed{D}+1)
    \cdots(\slashed{D}+N)
\end{align}
for all $N\in\N_0$, since the scalar curvature is $\tau=n(n-1)$ and so $J=\frac n2$. 
This agrees with the results
in \cite{Branson4, ES}.


\section{Application: Holographic deformation of the Dirac operator on Einstein manifolds}\label{Outlook}

   The inversion formula for GJMS operators $P_{2N}(g)$, $N\in\N$ ($N\leq\frac n2$ for even $n$), 
   cf.~\cite{Juhl1}, implies
   the existence of a sequence of second order linear differential operators $\mathcal{M}_{2N}$
   acting on functions and fulfilling 
  \begin{align}
     P_{2N}(g)\in\N[\mathcal{M}_2,\ldots,\mathcal{M}_{2N}],\nonumber\\
     \mathcal{M}_{2N}\in\Z[P_2(g),\ldots,P_{2N}(g)].
  \end{align} 
  Let us define a sequence of first order differential operators acting on the spinor bundle $S(M,h)$: 
  \begin{align}
      M_{2N+1}:=(-1)^N(N!)^2\left(\tfrac{2J}{n}\right)^N\slashed{D},\, N\in\N_0 \label{eq:TheMs} .
  \end{align}
  Notice that
  the operators $M_{2N+1}$ are deformations of the Dirac operator. 
  \begin{theorem}
    Let $(M,h)$ be a semi-Riemannian Einstein {\it Spin}-manifold of dimension $n$. 
    For all $N\in\N_0$, we have
    \begin{align}
      \mathcal{D}_{2N+1}\in\N[M_1,\ldots,M_{2N+1}],\nonumber\\
      M_{2N+1}\in\Z[\mathcal{D}_{1},\ldots,\mathcal{D}_{2N+1}].
    \end{align}
  \end{theorem}
  \begin{proof}
     We prove, by induction, that for all $N\in\N_0$
    we have $\mathcal{D}_{2N+1}\in\N[M_1,M_3,\ldots,\break M_{2N+1}]$, 
    and the coefficient of $M_{2N+1}$ equals 
    $1$. 
    The case $N=0$ is 
    obvious by definition. Let us assume 
    that $\mathcal{D}_{2N-1}\in\N[M_1,\ldots,M_{2N-1}]$ such that 
    the coefficient of $M_{2N-1}$ equals 
    $1$. By 
    Theorem~\ref{ProductOfConformalPowers}, we have
    \begin{equation}
         \mathcal{D}_{2N+1}= \mathcal{D}_{2N-1}\left(\slashed{D}^2-N^2\left(\tfrac{2J}{n}\right)\right)
         = \mathcal{D}_{2N-1}M_1^2-N^2\left(\tfrac{2J}{n}\right)\mathcal{D}_{2N-1}. 
    \end{equation}
    Since $\mathcal{D}_{2N-1}=B+M_{2N-1}$ for some $B\in \N[M_1,\ldots,M_{2N-3}]$, 
    and $B$ contains in each contribution at least one $M_1$, we can absorb the factor 
    $-\frac{2J}{n}$ into $M_1$ to get a contribution of $M_3$. To finish the 
    proof, we note that 
    $M_{2N+1}=-N^2(\frac{2J}{n})M_{2N-1}$, hence 
    $\mathcal{D}_{2N+1}\in\N[M_1,M_3,\ldots, M_{2N+1}]$ 
    with the coefficient of $M_{2N+1}$ being $1$. 
    Notice that we can not expect a unique 
    polynomial expression for $\mathcal{D}_{2N+1}$. 
  \end{proof}
    
  Closely related to the sequence \eqref{eq:TheMs} 
  is the following generating function termed holographic deformation of the Dirac operator,
  \begin{align}
       \slashed{\mathcal{H}}(r):=\sum_{N\geq 0}\tfrac{(-1)^N}{(N!)^2}\left(\tfrac{r}{2}\right)^{2N}M_{2N+1},
  \end{align}
   a deformation of the Dirac operator on $M$ in the sense that $\slashed{\mathcal{H}}(0)=\slashed{D}$. 
   It has the holographic description 
  \begin{align}
      \slashed{\mathcal{H}}(r)\psi
         =\sum_{i=1}^n\varepsilon_i \sqrt{h_r^{-1}}(s_i)\cdot\nabla^{h,S}_{s_i}\psi,\label{eq:Holo}
  \end{align} 
  where $\psi\in\Gamma\big(S(M,h)\big)$, $\{s_i\}_{i=1}^{n}$ is an $h$-orthonormal frame and $\sqrt{h_r^{-1}}$ 
  is a formal power series in $r$, 
  \begin{align*}
      \sqrt{h_r^{-1}}=\left(1-\tfrac{J}{2n}r^2\right)^{-1}h=\sum_{N\geq 0}\left(\tfrac{J}{2n}r^2\right)^Nh. 
  \end{align*}

  \begin{bem}
    The existence of $M_1$, $M_3$ and $M_5$ is shown for general curved 
    manifolds in \cite[Chapter~$6$]{Fischmann}, and reduces in the case of Einstein manifolds 
    to sequence \eqref{eq:TheMs}.
    Notice that the holographic description \eqref{eq:Holo} reproduces the first order contributions 
    of $M_k$, $k=1,3,5$. The structure of constant (zeroth order) terms of $M_k$, $k=3,5$, remains unclear.
  \end{bem}

\vspace{0.5cm}

\bibliographystyle{amsalpha}
\bibliography{bibliography}

\end{document}